\documentclass[11pt,twoside]{amsart}
\usepackage{mathrsfs}
\usepackage{amsmath}
\usepackage{amsthm}

\usepackage{amsfonts}
\usepackage{amssymb}
\usepackage{latexsym}
\usepackage[all]{xy}

\allowdisplaybreaks[4] \footskip=15pt
\renewcommand{\uppercasenonmath}[1]{}

\numberwithin{equation}{section} \theoremstyle{plain}
\newtheorem*{thm*}{Main Theorem}
\newtheorem{theorem}{Theorem}[section]
\newtheorem{corollary}[theorem]{Corollary}
\newtheorem*{corollary*}{Corollary}
\newtheorem{lemma}[theorem]{Lemma}
\newtheorem*{lemma*}{Lemma}
\newtheorem{proposition}[theorem]{Proposition}
\newtheorem*{proposition*}{Proposition}
\newtheorem{remark}[theorem]{Remark}
\newtheorem*{remark*}{Remark}
\newtheorem{definition}[theorem]{Definition}
\newtheorem*{definition*}{Definition}

\newtheorem*{acknowledgements*}{ACKNOWLEDGEMENTS}

\newcommand{\Ext}{\mbox{\rm Ext}}
\newcommand{\Hom}{\mbox{\rm Hom}}
\newcommand{\Tor}{\mbox{\rm Tor}}
\newcommand{\im}{\mbox{\rm im}}

\newcommand{\pd}{\mbox{\rm pd}}

\begin{document}
\begin{center}
{\large  \bf STABILITY OF GORENSTEIN FLAT CATEGORIES WITH RESPECT TO
A SEMIDUALIZING MODULE \footnote{This work is supported by the
National Natural Science Foundation of China (10971024), the Natural
Science Foundation of Jiangsu Province (BK2010393) and the
Scientific Research Foundation of Graduate School of Jiangsu
province (CXZZ12-0082).}}\\ \vspace{0.8cm} {\small \bf ZHENXING
DI$^{1}$\footnote{Corresponding author. E-mail:
dizhenxing19841111@126.com, liuzk@nwnu.edu.cn, jlchen@seu.edu.cn.},
ZHONGKUI LIU$^{2}$ and JIANLONG CHEN$^{1}$
}\\ \vspace{0.6cm} {\rm $^{1}$Department of Mathematics, Southeast University, Nanjing, 210096, P. R. CHINA\\
$^{2}$College of Mathematics and Information Science, Northwest
Normal University, Lanzhou 730070, P. R. China}
\end{center}

\bigskip
{ \bf  Abstract:}  \leftskip0truemm\rightskip0truemm In this paper,
we first introduce $\mathcal {W}_F$-Gorenstein modules to establish
the following Foxby equivalence:
\begin{center}
$\xymatrix@C=80pt{\mathcal {G}(\mathcal {F})\cap \mathcal {A}_C(R)
\ar@<0.5ex>[r]^{C\otimes_R-} & \mathcal {G}(\mathcal {W}_F)
   \ar@<0.5ex>[l]^{\textrm{Hom}_R(C,-)}} $
\end{center}
where $\mathcal {G}(\mathcal {F})$, $\mathcal {A}_C(R)
   $ and $\mathcal {G}(\mathcal {W}_F)$
denote the class of Gorenstein flat modules, the Auslander class and
the class of $\mathcal {W}_F$-Gorenstein modules respectively. Then,
we investigate two-degree $\mathcal {W}_F$-Gorenstein modules. An
$R$-module $M$ is said to be two-degree $\mathcal {W}_F$-Gorenstein
if there exists an exact sequence
\begin{center}
$\mathbb{G}_\bullet=\indent \cdots\longrightarrow G_1\longrightarrow
G_0\longrightarrow G^0\longrightarrow G^1\longrightarrow\cdots$
\end{center}
in $\mathcal {G}(\mathcal
{W}_F)$ such that $M \cong$ $\im(G_0\rightarrow G^0) $ and that
$\mathbb{G}_\bullet$ is Hom$_R(\mathcal {G}(\mathcal {W}_F),-)$ and
$\mathcal {G}(\mathcal {W}_F)^+\otimes_R-$ exact. We show that two
notions of the two-degree $\mathcal {W}_F$-Gorenstein and the
$\mathcal {W}_F$-Gorenstein modules coincide when R is a commutative
GF-closed ring.
 \bigskip


\section { \bf Introduction}
\indent Throughout this article, $R$ is a commutative ring with
identity and all modules are unitary. We denote by $R$-Mod the
category of $R$-modules. For an $R$-module $M$, the Pontryagin dual
or character module Hom$_\mathbb{Z}(M,\mathbb{Q}/\mathbb{Z})$ is
denoted by $M^+$.

Recall from White [16] that an $R$-module $C$ is said to be
semidualizing if $C$ admits a degreewise finite projective
resolution, the natural homothety map $R\rightarrow$ Hom$_R(C,C)$ is
an isomorphism and Ext$_R^{\geqslant1} (C,C)=$ 0. Examples include
the rank one free module and a dualizing (canonical) module when one
exists. With this notion, the Auslander class and the Bass class
with respect to a semidualizing $R$-module $C$, denoted by $\mathcal
{A}_C(R)$ and $\mathcal {B}_C(R)$ respectively, can be defined and
studied naturally. It is well known that there exists the following
equivalence of categories:
\begin{center}$\xymatrix@C=80pt{\mathcal
{A}_C(R) \ar@<0.5ex>[r]^{C\otimes_R-} & \mathcal {B}_C(R)
   \ar@<0.5ex>[l]^{\textrm{Hom}_R(C,-)}}$
\end{center}

Recently, to be a kind of generalization of the classes of
Gorenstein projective and Gorenstein injective modules, denoted by
$\mathcal {G}(\mathcal {P})$ and $ \mathcal {G}(\mathcal {I}) $
respectively, Geng and Ding [9] introduced the notions of the
$\mathcal {W}_P$-Gorenstein and the $\mathcal {W}_I$-Gorenstein
modules. Moreover, they obtained the following interesting
equivalences of categories:
\begin{center}$\xymatrix@C=80pt{\mathcal {G}(\mathcal {P})\cap \mathcal {A}_C(R) \ar@<0.5ex>[r]^{C\otimes_R-} & \mathcal {G}(\mathcal {W}_P)
   \ar@<0.5ex>[l]^{\textrm{Hom}_R(C,-)} ~~\textrm{and}~~\mathcal {G}(\mathcal {W}_I) \ar@<0.5ex>[r]^{~~~C\otimes_R-} & \mathcal {G}(\mathcal {I})\cap \mathcal {B}_C(R)
   \ar@<0.5ex>[l]^{~~~~\textrm{Hom}_R(C,-)} }$\end{center}
where $\mathcal {G}(\mathcal {W}_P)$ and $\mathcal {G}(\mathcal
{W}_I)$ denote the classes of $\mathcal {W}_P$-Gorenstein and
$\mathcal {W}_I$-Gorenstein modules respectively. So it is naturally
to ask if there exist some other classes satisfying the following
diagram:
\begin{center}$\xymatrix@C=80pt{\mathcal {G}(\mathcal {F})\cap \mathcal {A}_C(R) \ar@<0.5ex>[r]^{C\otimes_R-}
&~~~ \textrm{?}
   \ar@<0.5ex>[l]^{\textrm{Hom}_R(C,-)}}$
\end{center}
The motivation of the present article is the ※$\textrm{?}$§.

We shall introduce, in section 3, the notion of the $\mathcal
{W}_F$-Gorenstein modules, which plays the role of ※$\textrm{?}$§.
Combined with $\mathcal {W}_P$-Gorenstein and $\mathcal
{W}_I$-Gorenstein modules, they can be treated from a similar aspect
as the relationship among projective, injective and flat modules in
classical homological algebra theory. An $R$-module $M$ is said to
be $\mathcal {W}_F$-Gorenstein if there exists an exact sequence
\begin{center}$\mathbb{W}_\bullet =\indent \cdots\longrightarrow W_1\longrightarrow
W_0\longrightarrow W^0\longrightarrow W^1\longrightarrow\cdots$
\end{center}
in $\mathcal {F}_C(R)$ such that $M \cong$
$\im(W_0\rightarrow W^0) $ and that $\mathbb{W}_\bullet$ is
Hom$_R(\mathcal {P}_C(R),-)$ and $ \mathcal {I}_C(R)\otimes_R-$
exact, where $\mathcal {F}_C(R)$, $\mathcal {P}_C(R)$ and $\mathcal
{I}_C(R)$ denote the classes of $C$-flat, $C$-projective and
$C$-injective modules respectively. Furthermore, we get the
following theorem demonstrating the relationship between the classes
$\mathcal {G}(\mathcal {W}_F)$ and $\mathcal {G}\mathcal {F}_C(R)$
(see Theorem 3.4):\vspace{2mm}
\\$\textbf{Theorem~I}$\quad Let $C$ be a semidualizing $R$-module.
Then we have $\mathcal {G}(\mathcal {W}_F)= \mathcal {G}\mathcal
{F}_C(R)\cap \mathcal {B}_C(R)$.\vspace{2mm}\\ Also, the $\mathcal
{G}(\mathcal {W}_F)$-projective dimension for any $R$-module will be
investigated in this section.

In section 4, we first introduce the modules that arise from an
iteration of the above construction. To wit, let $\mathcal
{G}^2(\mathcal {W}_F)$ denote the class of $R$-module $M$ for which
there exists an exact sequence
\begin{center}
$  \mathbb{G}_\bullet =  \cdots\longrightarrow
G_1\longrightarrow G_0\longrightarrow G^0\longrightarrow
G^1\longrightarrow\cdots$
\end{center}
in $\mathcal {G}(\mathcal
{W}_F)$ such that $M \cong$ $\im(G_0\rightarrow G^0) $ and that
$\mathbb{G}_\bullet$ is Hom$_R(\mathcal {G}(\mathcal {W}_F),-)$ and
$\mathcal {G}(\mathcal {W}_F)^+\otimes_R-$ exact. Similarly, we can
also define $R$-modules which belong to $\mathcal {G}^2(\mathcal
{G}\mathcal {F}_C(R)\cap \mathcal {B}_C(R))$ or $\mathcal
{G}^2(\mathcal {F})$. Although the definition defined above differs
from the one appeared in [14], there is still a good correspondence.
We then apply those techniques obtained in the former parts of this
paper to get our results concerning the stability properties of
Gorenstein categories (see Theorem 4.5,
Corollary 4.6 and Corollary 4.7).\vspace{2mm}\\
$\textbf{Theorem~II}$\quad Let $R$ be a GF-closed ring and $C$ be a
semidualizing $R$-module. Then the following hold:

(1) $\mathcal {G}^2(\mathcal {W}_F)=\mathcal {G}(\mathcal {W}_F)$.

(2) $\mathcal {G}^2(\mathcal {G}\mathcal {F}_C(R)\cap \mathcal
{B}_C(R))=\mathcal {G}\mathcal {F}_C(R)\cap \mathcal {B}_C(R)$.

(3) $\mathcal {G}^2(\mathcal {F})=\mathcal {G}(\mathcal {F})$.
\vspace{2mm}

In the following part of this paper, let $C$ be a semidualizing
$R$-module and we mainly recall some necessary notions and
definitions in the next section. \vspace{3mm}


\section { \bf Notions and definitions}
Let $\mathcal{X}=\mathcal{X}(R)$, $\mathcal{Y}=\mathcal{Y}(R) $ be
classes of $R$-modules. We begin with the following definition.

\begin{definition}
We write $\mathcal {X}~\bot~\mathcal {Y}$ if
$\Ext^{\geqslant1}_R(X,Y)=0$ for each object $X\in\mathcal {X} $ and
each object $Y\in\mathcal {Y}$. Write $\mathcal {X}~\top~\mathcal
{Y}$ if $\Tor^R_{\geqslant1}(X,Y)=0$ for each object $X\in\mathcal
{X} $ and each object $Y\in\mathcal {Y}$. For an $R$-module $M$,
write $M~\bot~\mathcal {Y}$ if $\Ext^{\geqslant1}_R(M,Y)=0$ (resp.,
$\mathcal {Y}~\bot~M$ if $\Ext^{\geqslant1}_R(Y,M)=0$) for each
object $Y\in\mathcal {Y}$. Write $\mathcal {Y}~\top~M$ if
$\Tor^R_{\geqslant1}(Y,M)=0$ for each object $Y\in\mathcal {Y}$.
Following [15], we say that $\mathcal {X}$ is a generator for
$\mathcal {Y}$ if $\mathcal {X}\subseteq\mathcal {Y}$ and, for each
object $Y\in\mathcal {Y}$, there exists a short exact sequence
$$0\longrightarrow Y'\longrightarrow X\longrightarrow Y\longrightarrow0$$
in $\mathcal {Y}$ such that $X\in\mathcal {X}$. The class $\mathcal
{X}$ is a projective generator for $\mathcal {Y}$ if $\mathcal {X}$
is a generator for $\mathcal {Y}$ and $\mathcal {X}~\bot~\mathcal
{Y}$.
\end{definition}

\begin{definition}
For any $R$-module $M$, we recall three types of resolutions.

(1) [10, 1.5] A left $\mathcal {X}$-resolution of $M$ is an exact
sequence $\mathbb{X}$ = $ \cdots\rightarrow X_1\rightarrow
X_0\rightarrow M \rightarrow  0 $ with $X_n \in \mathcal {X}$ for
all $n \geqslant 0$.

(2) [10, 1.5] A right $\mathcal {X}$-resolution of $M$ is an exact
sequence $\mathbb{X}$ = $ 0\rightarrow M\rightarrow X^0\rightarrow
X^1 \rightarrow  \cdots $ with $X^n \in \mathcal {X}$ for all
$n\geqslant 0$.

Now let $\mathbb{X}$ be any (left or right) $\mathcal
{X}$-resolution of $M$. We say that $\mathbb{X}$ is co-proper if the
sequence $\Hom_R(\mathbb{X},Y)$ is exact for all $Y \in \mathcal
{X}$.

(3) [16, 1.6] A degreewise finite projective (resp., free)
resolution of $M$ is a left projective (resp., free) resolution
$\mathbb{P}$ of $M$ such that each $P_i$ is finitely generated
projective (resp., free). It is easy to verify that $M$ admits a
degreewise finite projective resolution if and only if $M$ admits a
degreewise finite free resolution.
\end{definition}

\begin{definition}
The $\mathcal {X}$-projective dimension of an $R$-module $M$ is
defined as
\begin{center}$\mathcal {X}\textrm{-}\pd_R(M)= \inf \{ \sup \{n\mid X_n \neq 0\}\mid
\mathbb{X}~is~ a~ left~\mathcal {X}\textrm{-}resolution~ of~
M\}.$\end{center}
Dually, we can also define the $\mathcal
{X}$-injective dimension of $M$.
\end{definition}

The next lemma has a standard proof.

\begin{lemma}
Let $M$ be an $R$-module. Consider the following exact sequence in
$\mathcal {X}$:
\begin{center}$\mathbb{X}=~~ \xymatrix@C=0.5cm{
  \cdots \ar[rr] &&  X_1 \ar[rr]^{\delta^\mathbb{X}_1} && X_0  \ar[rr]^{\delta^\mathbb{X}_0} && X_{-1} \ar[rr]&& \cdots
}$\end{center}
Then we have the following hold:

(1) Assume $ M~\bot~ \mathcal {X}$. If $\mathbb{X}$ is $\Hom_R(M,-)$
exact, then $\Ext^{\geqslant1}_R(M,\im(\delta^\mathbb{X}_i
))=0$ for all $i$. Conversely, if
$\Ext^1_R(M,\im(\delta^\mathbb{X}_i))=0$ for all $i$, then
$\mathbb{X}$ is $\Hom_R(M,-)$ exact.

(2) Assume $M~\top~ \mathcal {X}$. If $\mathbb{X}$ is $M\otimes_R -$
exact, then $\Tor^R_{ \geqslant1}(M,\im(\delta^\mathbb{X}_i
)) = 0$ for all $i$. Conversely, if
$\Tor^R_1(M,\im(\delta^\mathbb{X}_i) ) = 0$ for all $i$, then
$\mathbb{X}$ is $M\otimes_R- $ exact.
\end{lemma}

\begin{definition}
An $R$-module $M$ is said to be Gorenstein flat if there exists an
exact sequence
\begin{center}
$\mathbb{X}=  \cdots\longrightarrow F_1\longrightarrow
F_0\longrightarrow  F^0\longrightarrow
F^1\longrightarrow\cdots$\end{center}
in $R$-Mod with each $F_i$ and
$F^i$ flat such that $M\cong$ $\im(F_0\rightarrow  F^0)$ and that
$\mathbb{X}$ is $I\otimes_R-$ exact for any injective $R$-module
$I$. The exact sequence $\mathbb{X}$ is called complete flat
resolution of $M$.
\end{definition}

In the following, we denote the class of Gorenstein flat modules by
$\mathcal{G}(\mathcal {F})$.

\begin{definition}
Let $R$ be a ring. We call $R$ GF-closed if the class of Gorenstein
flat $R$-modules is closed under extensions, that is, if $ 0
\rightarrow X \rightarrow Y \rightarrow Z \rightarrow0$ is an exact
sequence with $X$ and $Z$ Gorenstein flat modules, then $Y$ is also
Gorenstein flat.
\end{definition}

It follows from [2] that the class of GF-closed rings includes
strictly the one of coherent rings and the one of rings of finite
weak global dimension.

\begin{definition}
An $R$-module is $C$-projective (resp., $C$-flat) if it has the form
$C \otimes_R P$ for some projective (resp., flat) $R$-module $P$. An
$R$-module is $C$-injective if it has the form $\Hom_R(C,I)$ for some
injective $R$-module $I$. We set\\
$\indent\indent\indent\indent\indent\indent\indent\indent \mathcal
{P}_C(R) = \{C\otimes_R
P\mid P$ is a projective $R$-module$\}$\\
$\indent\indent\indent\indent\indent\indent\indent\indent\mathcal
{F}_C(R) = \{C\otimes_R F\mid
F$ is a flat $R$-module$\}$\\
$\indent\indent\indent\indent\indent\indent\indent\indent\mathcal
{I}_C(R) = \{$$\Hom_R(C,I)\mid I$ is an injective $R$-module$\}.$
\end{definition}

\begin{remark}
The classes defined above are studied extensively in [12]. From
there we know that

(1) The classes $\mathcal {F}_C(R)$ and $\mathcal {P}_C(R)$ are
closed under arbitrary direct sums and summands and if $R$ is
coherent, then $\mathcal {F}_C(R)$ is also closed under arbitrary
direct products.

(2) The class $\mathcal {I}_C(R)$ is closed under arbitrary direct
products and summands.
\end{remark}

\begin{definition}
An $R$-module $N$ is said to be $G_C$-injective ($G_C$-inj for
short) if there exists an exact sequence
\begin{center}$\indent\mathbb{Y} =
\cdots\longrightarrow \Hom_R(C,I^1)\longrightarrow
\Hom_R(C,I^0)\longrightarrow
 I_0\longrightarrow I_1\longrightarrow\cdots$
 \end{center}
in $R$-Mod with each $I_i$ and $I^i$ injective such that $N\cong$
$\im(\Hom_R(C,I^0)\rightarrow I_0)$ and that $\mathbb{Y}$ is
$\Hom_R(\mathcal {I}_C(R),-)$ exact. The exact sequence $\mathbb{Y}$
is called complete $\mathcal {I}_C\mathcal {I}$-resolution of $N$.

An $R$-module $T$ is said to be $G_C$-flat if there exists an exact
sequence
\begin{center}$\mathbb{Z}= \indent \cdots\longrightarrow
F_1\longrightarrow F_0\longrightarrow C\otimes_R F^0\longrightarrow
C\otimes_R F^1\longrightarrow\cdots$
\end{center}
in $R$-Mod with
each $F_i$ and $F^i$ flat such that $M\cong$ $\im(F_0\rightarrow
C\otimes_R F^0)$ and that $\mathbb{Z}$ is $\mathcal {I}_C(R)\otimes_R-$
exact. The exact sequence $\mathbb{Z}$ is called complete $\mathcal
{F}\mathcal {F}_C$-resolution of $T$.
\end{definition}

We will denote the classes of $G_C$-inj and $G_C$-flat modules by
$\mathcal {G}\mathcal {I}_C(R)$ and $\mathcal {G}\mathcal {F}_C(R)$
respectively.

\begin{remark}
Similar to the proofs in [16] we can easily get that

(1) Every $C$-injective $R$-module is $G_C$-inj and the class
$\mathcal {G}\mathcal {I}_C(R)$ is coresolving and closed under
arbitrary direct products and summands.

(2) Every $C$-flat $R$-module is $G_C$-flat and the class $\mathcal
{G}\mathcal {F}_C(R)$ is closed under arbitrary direct sums.

(3) Every kernel and cokernel of a complete $\mathcal {I}_C\mathcal
{I}$-resolution (resp., $\mathcal {F}\mathcal {F}_C$-resolution)
belongs to $\mathcal {G}\mathcal {I}_C(R)$ (resp., $\mathcal
{G}\mathcal {F}_C(R)$).
\end{remark}

By using the definition of $G_C$-flat modules and Remark 2.10, the
proof of the next lemma is a standard argument.\vspace{3mm}

\begin{lemma}
 The following are equivalent for an $R$-module
$M$:

(1) $M$ is $G_C$-flat.

(2) $M$ satisfies the following two conditions:

$\indent(i)$ $\mathcal {I}_C(R)~\top~M$ and

$\indent(ii)$ There exists an exact sequence $0\rightarrow
M\rightarrow C\otimes_RF^0\rightarrow C\otimes_RF^1\rightarrow
\cdots$ in $R$-Mod with each $F^i$ flat such that $\mathcal
{I}_C(R)\otimes_R-$ leaves it exact.

(3) There exists a short exact sequence $0\rightarrow M\rightarrow
C\otimes_RF\rightarrow G\rightarrow0$ in $R$-Mod with $F$ flat and
$G\in\mathcal {G}\mathcal {F}_C(R)$.
\end{lemma}

\begin{definition}
The Auslander class $\mathcal {A}_C(R)$ with respect to $C$ consists
of all $R$-modules $M$ satisfying

(1) $\Tor^R_{\geqslant 1}(C,M)=0=$ $\Ext_R^{\geqslant 1}(C,C\otimes_R
M)$ and

(2) The natural evaluation map $\mu_{CCM}:M\rightarrow
$$\Hom_R(C,C\otimes_R M)$ is an isomorphism.\\
Dually, the Bass class $\mathcal {B}_C(R)$ with respect to $C$
consists of all $R$-modules $N$ satisfying

(1) $\Ext_R^{\geqslant 1}(C,N) =0=$ $\Tor^R_{\geqslant
1}(C,$$\Hom_R(C,N))$ and

(2) The natural evaluation map $\nu_{CCN}
:C\otimes_R$$\Hom_R(C,N)\rightarrow N$ is an isomorphism.\vspace{3mm}
\end{definition}

We now state some basic results about the classes $\mathcal
{A}_C(R)$ and $\mathcal {B}_C(R)$.

\begin{lemma}[12]
The following hold:

(1) If any two $R$-modules in a short exact sequence are in
$\mathcal {A}_C(R)$, respectively $\mathcal {B}_C(R)$, then so is
the third.

(2) The class $\mathcal {A}_C(R)$ contains all modules of finite
flat dimension and those of finite $\mathcal {I}_C$-injective
dimension. The class $\mathcal {B}_C(R)$ contains all modules of
finite injective dimension and those of finite $\mathcal
{F}_C$-projective dimension.
\end{lemma}

To be a direct corollary of [12, Theorem 6.4], we have the following
lemma:

\begin{lemma}
$\mathcal {P}_C(R)~\bot~\mathcal {B}_C(R)$, $\mathcal
{A}_C(R)~\bot~\mathcal {I}_C(R)$ and $\mathcal
{A}_C(R)~\top~\mathcal {F}_C(R)$.
\end{lemma}

\section { \bf $\mathcal {W}_F$-Gorenstein modules}

Now we introduce and investigate the notion of $\mathcal
{W}_F$-Gorenstein modules and the corresponding projective dimension
for any $R$-Module.

\begin{definition}
An $R$-module $M$ is said to be $\mathcal {W}_F$-Gorenstein if there
exists an exact sequence
\begin{center}
$\mathbb{W}_\bullet =  \cdots\longrightarrow
W_1\longrightarrow W_0\longrightarrow W^0\longrightarrow
W^1\longrightarrow\cdots$
\end{center}
in $\mathcal {F}_C(R)$ such
that $M \cong$ $\im(W_0\rightarrow W^0) $ and that $\mathbb{W}_\bullet$ is
$\Hom_R(\mathcal {P}_C(R),-)$ and $ \mathcal {I}_C(R)\otimes_R-$
exact.
\end{definition}

It is clear that each module in $\mathcal {F}_C(R)$ is $\mathcal
{W}_F$-Gorenstein, and every kernel and cokernel of
$\mathbb{W}_\bullet$ is $\mathcal {W}_F$-Gorenstein.

In the following, we denote by $\mathcal {G}(\mathcal {W}_F)$ the
class of $\mathcal {W}_F$-Gorenstein modules.

\begin{proposition}
$\mathcal {P}_C(R)~\bot~\mathcal {G}(\mathcal {W}_F)$ and $\mathcal
{I}_C(R)~\top~\mathcal {G}(\mathcal {W}_F)$.
\end{proposition}
\begin{proof}
It follows directly from Lemma 2.4 and Lemma 2.14.
\end{proof}

\begin{proposition}
Let $ \mathbb{W}_\bullet=  \cdots\rightarrow W_1\rightarrow
W_0\rightarrow W^0\rightarrow W^1\rightarrow\cdots$ be an exact
sequence in $\mathcal {F}_C(R)$ and $M \cong$ $\im(W_0\rightarrow
W^0) $. Then $\mathbb{W}_\bullet $ is $\Hom_R(\mathcal {P}_C(R),-)$
exact if and only if $M\in \mathcal {B}_C(R)$.
\end{proposition}
\begin{proof}
Suppose that $M\in \mathcal {B}_C(R)$. By Lemma 2.13, every kernel
and cokernel of $\mathbb{W}_\bullet $ is in $\mathcal {B}_C(R)$, and
so $\mathbb{W}_\bullet $ is $\Hom_R({P}_C(R),-)$ exact by Lemma 2.4
and Lemma 2.14.

Conversely, if $\mathbb{W}_\bullet $ is $\Hom_R(\mathcal {P}_C(R),-)$
exact, then by Lemma 2.4 and Lemma 2.14, we have $\mathcal
{P}_C(R)~\bot ~M$. Thus, there exists an exact sequence
\begin{center}
$\cdots\longrightarrow W_1 \longrightarrow W_0 \longrightarrow I^0
\longrightarrow I^1 \longrightarrow \cdots$
\end{center}
in $R$-Mod with each $I^i$ injective such that $M \cong$
$\im(W_0\rightarrow I^0) $ and that $\Hom_R(\mathcal {P}_C(R),-)$ leaves it
exact. Hence $M\in \mathcal {B}_C(R)$ by [12, Theorem 6.1].
\end{proof}

Now we are in position to prove the result of linking the classes
$\mathcal {G}\mathcal {F}_C(R)$ and $\mathcal {G}(\mathcal {W}_F)$
together.

\begin{theorem}
Let $M$ be an $R$-module. Then $M\in \mathcal {G}(\mathcal {W}_F)$
if and only if $M\in\mathcal {G}\mathcal {F}_C(R)\cap \mathcal
{B}_C(R)$.
\end{theorem}
\begin{proof}
($\Rightarrow$) Let $M\in \mathcal {G}(\mathcal {W}_F)$. We first
have $\mathcal {I}_C(R)~\top~ M$ by Proposition 3.2. Then
$M\in\mathcal {G}\mathcal {F}_C(R)\cap \mathcal {B}_C(R)$ by Lemma
2.11 and Proposition 3.3.

($\Leftarrow$) Let $M\in\mathcal {G}\mathcal {F}_C(R)\cap \mathcal
{B}_C(R)$. Since $M\in\mathcal {G}\mathcal {F}_C(R)$, we have that
$\mathcal {I}_C(R)~\top~ M$ and there exists an exact sequence
 \begin{center}
 $  0\longrightarrow M \longrightarrow W^0 \longrightarrow W^1\longrightarrow \cdots
$
\end{center}
in $R$-Mod with each  $W^i \in \mathcal {F}_C(R)$ such that
$\mathcal {I}_C(R)\otimes_R- $ leaves it exact. On the other hand,
since $M \in\mathcal {B}_C(R) $, it is easy to verify that $M$ has a
proper left $\mathcal {P}_C(R)$-resolution
\begin{center}
$
  \cdots \longrightarrow V_1 \longrightarrow V_0 \longrightarrow M \longrightarrow
  0.
$\end{center}
It follows from Lemma 2.13 and Lemma 2.14 that
$\mathcal {I}_C(R)\otimes_R- $ leaves it exact. Thus, we have the
following exact sequence:
\begin{center}
$
  \cdots \longrightarrow V_1 \longrightarrow  V_0 \longrightarrow W^0 \longrightarrow W^1\longrightarrow \cdots
$\end{center}
such that $M \cong$ $\im(V_0\rightarrow W^0)$. By
Proposition 3.3, we know that $\Hom_R(\mathcal {P}_C(R),-)$ leaves it
exact. It follows that $M\in \mathcal {G}(\mathcal {W}_F)$.
\end{proof}

\begin{theorem}
There exists equivalence of categories:
\begin{center}
$\xymatrix@C=80pt{\mathcal {G}(\mathcal {F})\cap \mathcal {A}_C(R) \ar@<0.5ex>[r]^{~~~~~~~C\otimes_R-} &
~~ \mathcal {G}(\mathcal {W}_F).
   \ar@<0.5ex>[l]^{\indent\Hom_R(C,-)}} $
   \end{center}
\end{theorem}
\begin{proof}
We first show that the functor $\Hom_R(C,-)$ maps $\mathcal
{G}(\mathcal {W}_F)$ to $\mathcal {G}(\mathcal {F})\cap \mathcal
{A}_C(R)$. Assume that $M\in\mathcal {G}(\mathcal {W}_F)$. Then
there exists an exact sequence
\begin{center}
$\mathbb{W}_\bullet=
\cdots\longrightarrow W_1\longrightarrow W_0\longrightarrow
W^0\longrightarrow W^1\longrightarrow\cdots$
\end{center}
in $\mathcal {F}_C(R)$ such that $M \cong$ $\im(W_0\rightarrow W^0) $
and that $\mathbb{W}_\bullet$ is $\Hom_R(\mathcal {P}_C(R),-)$ and
$\mathcal {I}_C(R)\otimes_R-$ exact. So $M\in \mathcal {B}_C(R)$ by
Theorem 3.4, and hence every kernel and cokernel of
$\mathbb{W}_\bullet$ is in ${B}_C(R)$ by Lemma 2.13. Thus,
$\Hom_R(C,\mathbb{W}_\bullet)$ is exact, moreover,
$\Hom_R(C,M)\in \mathcal {A}_C(R)$ by [12, Proposition 4.1].
On the other hand, suppose that $W_i\cong C\otimes_RF_i$ and
$W^i\cong C\otimes_RF^i$, where each $F_i$ and $F^i$ flat. Then we
have the following exact sequence in $R$-Mod:
\begin{center}
$
\Hom_R(C,\mathbb{W}_\bullet) =\indent
  \cdots \longrightarrow F_1 \longrightarrow F_0 \longrightarrow F^0 \longrightarrow F^1 \longrightarrow \cdots
$
\end{center}
such that $\Hom_R(C,M) \cong
\im(F_0\rightarrow F^0)$. For each injective $R$-module $I$,
we have
\begin{center}
$I\otimes_R \Hom_R(C,\mathbb{W}_\bullet) \cong C\otimes_R
\Hom_R(C,I)\otimes_R\Hom_R(C,\mathbb{W}_\bullet)\cong
\Hom_R(C,I)\otimes_R \mathbb{W}_\bullet $
\end{center}
is exact. Hence, $\Hom_R(C,M)\in\mathcal {G}(\mathcal {F})$.

The proof of $C\otimes_R-$ maps $\mathcal {G}(\mathcal {F})\cap
\mathcal {A}_C(R)$ to $\mathcal {G}(\mathcal {W}_F)$ is similar.
\end{proof}

\begin{corollary}
Let $R$ be a GF-closed ring. Then the class $\mathcal {G}(\mathcal
{W}_F)$ is closed under extensions, kernels of epimorphisms and
direct summands.
\end{corollary}
\begin{proof}
We first show that the class $\mathcal {G}(\mathcal {W}_F)$ is
closed under extensions when $R$ is GF-closed. Consider the
following short exact sequence:
\begin{center}
$0\longrightarrow M\longrightarrow N\longrightarrow K\longrightarrow0$
\end{center}
with $M$ and $N$ belong to $\mathcal {G}(\mathcal {W}_F)$. Since
$M\in\mathcal {B}_C(R)$ by Theorem 3.4, we get the next exact
sequence
\begin{center}
$0\longrightarrow \Hom_R(C,M)\longrightarrow \Hom_R(C,N)\longrightarrow \Hom_R(C,K) \longrightarrow0.$
\end{center}
It follows from Theorem 3.5 that $\Hom_R(C,M)$ and
$\Hom_R(C,K)$ belong to $\mathcal {G}(\mathcal {F})\cap
\mathcal {A}_C(R)$. Thus, $\Hom_R(C,N)$ belongs to $\mathcal
{G}(\mathcal {F})\cap \mathcal {A}_C(R)$. On the other hand, since
$N\in\mathcal {B}_C(R)$ by Lemma 2.13 and Theorem 3.4, $N\cong
C\otimes_R$$\Hom_R(C,N)\in\mathcal {G}(\mathcal {W}_F)$ by Theorem
3.5, as desired.

The proof of the class $\mathcal {G}(\mathcal {W}_F)$ is closed
under kernels of epimorphisms and direct summands is similar to [2,
Theorem 2.3 and Corollary 2.6].
\end{proof}

\begin{lemma}
Let $R$ be a GF-closed ring. For every short exact sequence
$0\rightarrow G_1\rightarrow G_0\rightarrow M\rightarrow0$ in
$R$-Mod with $G_0,G_1\in\mathcal {G}(\mathcal {W}_F)$, if
$\Tor^R_1(\mathcal {I}_C(R),M)=0$, then $M\in\mathcal {G}(\mathcal
{W}_F)$.
\end{lemma}
\begin{proof}
By the fact that the class $\mathcal
 {F}_C(R)$ is closed direct summands and [14, Lemma 4.1], the proof of the lemma is similar to [2, Theorem 2.3].
\end{proof}

\begin{theorem}
Let $R$ be a GF-closed ring and $M$ an $R$-module with finite
$\mathcal {G}(\mathcal {W}_F)$-projective dimension and let $n$ be a
non-negative integer. Then the following are equivalent:

(1) $\mathcal {G}(\mathcal {W}_F)$-\pd$_R(M)\leqslant n$.

(2) For every non-negative integer $t$ such that $0\leqslant
t\leqslant n$, there exists an exact sequence $0\rightarrow
W_n\rightarrow \cdots\rightarrow W_1\rightarrow W_0\rightarrow
M\rightarrow0$ in $R$-Mod such that $W_t\in\mathcal {G}(\mathcal
{W}_F)$ and $W_i\in \mathcal {F}_C(R)$ for $i\neq t$.

(3) There exists a short exact sequence $0\rightarrow K\rightarrow G
\rightarrow M\rightarrow0$ in $R$-Mod with $G\in\mathcal
{G}(\mathcal {W}_F)$ and $\mathcal {F}_C(R)$-\pd$_R(K)\leqslant n-1$.

(4) There exists a short exact sequence $0\rightarrow M\rightarrow
K'\rightarrow G'\rightarrow0$ in $R$-Mod with $G'\in\mathcal
{G}(\mathcal {W}_F)$ and $\mathcal {F}_C(R)$-\pd$_R(K')\leqslant n$.

(5) There exists an exact sequence $0\rightarrow G\rightarrow
V_{n-1}\rightarrow\cdots\rightarrow V_0\rightarrow M\rightarrow0$ in
$R$-Mod with $G\in\mathcal {G}(\mathcal {W}_F)$ and $V_i \in
\mathcal {P}_C(R)$ for all $0\leqslant i \leqslant n-1$.

(6) For every exact sequence $0 \rightarrow K_n \rightarrow
G_{n-1}\rightarrow \cdots \rightarrow G_0 \rightarrow M \rightarrow
0$ in $R$-Mod with $G_i \in\mathcal {G}(\mathcal {W}_F)$ for all
$0\leqslant i\leqslant n-1$, then also $K_n\in\mathcal {G}(\mathcal
{W}_F)$.

(7) $\Tor^{n+j}_R(U,M)=0$ for all $j\geqslant1$ and all $U\in
\mathcal {I}_C(R)$.

(8) $\Tor^{n+j}_R(U,M)=0$ for all $j\geqslant1$ and all $U$ with
$\mathcal {I}_C(R)$-id$_R(U)<\infty$. \vspace{3mm}

Furthermore, we have that\vspace{2mm}

$\mathcal {G}(\mathcal {W}_F)\textrm{-\pd}_R(M) =
\sup\{~n\in\mathbb{N}~|~\Tor^n_R(U,M)\neq
0~\textrm{for some}~U\in \mathcal {I}_C(R)\}$\\
$\indent\indent\indent\indent\indent\indent\indent=
\sup\{~n\in\mathbb{N}~|~ \Tor^n_R(U,M)\neq
0~\textrm{for some}~U~\textrm{with}~\mathcal
{I}_C(R)\textrm{-id}_R(U)<\infty\}.$
\end{theorem}
\begin{proof}
(2) $\Rightarrow$ (3) $\Rightarrow$ (1) and (6) $\Rightarrow$ (1)
are clear.

(1) $\Rightarrow$ (7) $\Rightarrow$ (8) follow from usual dimension
shifting argument.

(1) $\Rightarrow$ (2) Since the class $\mathcal {G}(\mathcal {W}_F)$
is closed under extensions by Corollary 3.6, the proof is similar to
[13, Theorem 2.6].

(3) $\Rightarrow$ (4) Since $G \in \mathcal {G}(\mathcal {W}_F)$,
there exist a short exact sequence $  0 \rightarrow G \rightarrow W
\rightarrow G' \rightarrow 0$ in $R$-Mod with $W\in\mathcal
{F}_C(R)$ and $G' \in \mathcal {G}(\mathcal {W}_F)$. Now consider
the following push-out diagram:
\begin{center}
$\xymatrix{
     & & & 0\ar[d]_{}  & 0 \ar[d]_{} &  \\
     & 0\ar[r]& K\ar@{=}[d]^{} \ar[r]& G\ar[d]\ar[r] &M\ar[d]\ar[r]&0 \\
& 0\ar[r]& K \ar[r]& W\ar[d]\ar[r] &K'\ar[d]\ar[r]&0 \\
&&&G'\ar[d]_{} \ar@{=}[r]^{} & G' \ar[d]_{} \\
    && & 0& 0  &
      }$
\end{center}
From the second row in the above diagram, we know $\mathcal
{F}_C(R)$-\pd$_R(K')\leqslant n$. So the third column is as desired.

(4) $\Rightarrow$ (3) Since $\mathcal {F}_C(R)$-\pd$_R(K')\leqslant
n$, there exist a short exact sequence $  0 \rightarrow K
\rightarrow W \rightarrow K' \rightarrow 0$ in $R$-Mod with
$W\in\mathcal {F}_C(R)$ and $\mathcal {F}_C(R)$-\pd$_R(K)\leqslant
n-1$. Then consider the following pullback diagram:
\begin{center}
 $\xymatrix{
     &&  0\ar[d]_{}  & 0 \ar[d]_{} &  \\
      && K\ar@{=}[r]^{} \ar[d]& K\ar[d] & \\
    & 0 \ar[r]^{} & G\ar[d]_{} \ar[r]^{} & W\ar[d]_{} \ar[r]^{} & G' \ar@{=}[d]_{} \ar[r]^{} & 0  \\
      &0  \ar[r]^{} & M \ar[d]\ar[r]^{} &K'\ar[d]_{} \ar[r]^{} & G'  \ar[r]^{} & 0  \\
    & & 0& 0  &
      }$
\end{center}
From the second row, we know that $G\in\mathcal {G}(\mathcal {W}_F)$
by Corollary 3.6. So the first column is as desired.

(1) $\Rightarrow$ (5) It suffices to prove the case for $n=1$.
Assume that $\mathcal {G}(\mathcal {W}_F)$-\pd$_R(M)\leqslant 1$.
Then there exists a short exact sequence $
  0 \rightarrow G_1 \rightarrow G_0 \rightarrow M \rightarrow 0$ in $R$-Mod with $G_0,G_1\in\mathcal {G}(\mathcal {W}_F)$.
By Theorem 3.4, we know that $G_0\in\mathcal {B}_C(R)$. Thus, it is
easy to verify that there exists a short exact sequence $
  0 \rightarrow G_0'  \rightarrow V \rightarrow G_0 \rightarrow 0
$ in $R$-Mod such that $V\in\mathcal {P}_C(R)$, then also
$V\in\mathcal {G}(\mathcal {W}_F)$. By Corollary 3.6, we have $G_0'
\in\mathcal {G}(\mathcal {W}_F)$. Now consider the following
pullback diagram:
\begin{center}
 $\xymatrix{
     &&  0\ar[d]_{}  & 0 \ar[d]_{} &  \\
      && G_0'\ar@{=}[r]^{} \ar[d]& G_0'\ar[d] & \\
    & 0 \ar[r]^{} & G\ar[d]_{} \ar[r]^{} & V\ar[d]_{} \ar[r]^{} & M \ar@{=}[d]_{} \ar[r]^{} & 0  \\
      &0  \ar[r]^{} & G_1 \ar[d]\ar[r]^{} &G_0\ar[d]_{} \ar[r]^{} & M  \ar[r]^{} & 0  \\
    & & 0& 0  &
      }$
\end{center}
From the first column in the above diagram, we know that $G
\in\mathcal {G}(\mathcal {W}_F)$ by Corollary 3.6. So the middle row
is as desired.

(5) $\Rightarrow$ (6) Let $0 \rightarrow K_n \rightarrow
G_{n-1}\rightarrow \cdots \rightarrow G_0 \rightarrow M \rightarrow
0$ be an exact sequence in $R$-Mod with each $G_i\in\mathcal
{G}(\mathcal {W}_F)$, then also $ G_i \in\mathcal {B}_C(R)$ by
Theorem 3.4. Thus, $K_n \in\mathcal {B}_C(R)$ and $\mathcal
{P}_C(R)~\bot~ K_n$ by Lemma 2.13 and Lemma 2.14. Then we have the
following commutative diagram with exact rows:
\begin{center}
$\xymatrix{
  0 \ar[r]^{} &G_n  \ar[d]_{} \ar[r]^{} & V_{n-1}\ar[d]_{}
  \ar[r]^{} & \cdots   \ar[r]^{} &V_1  \ar[d]_{} \ar[r]^{} &V_0
  \ar[d]_{} \ar[r]^{} & M \ar@{=}[d] \ar[r]^{} & 0  \\
 0  \ar[r]^{} &K_n  \ar[r]^{} & G_{n-1}\ar[r]^{} & \cdots \ar[r]^{} &
  G_1 \ar[r]^{} &  G_0 \ar[r]^{} & M  \ar[r]^{} & 0    }
$
\end{center}
Thus, the mapping cone
\begin{center}
$
  0\longrightarrow G_n \longrightarrow V_{n-1}\oplus K_n \longrightarrow \cdots \longrightarrow V_0\oplus G_1 \longrightarrow G_0 \longrightarrow
  0
$
\end{center}
is exact. It follows from Corollary 3.6 that $K_n \in
\mathcal {G}(\mathcal {W}_F)$.

(8) $\Rightarrow$ (1) By Lemma 3.7, the proof is similar to [2,
Theorem 2.8].

The last claim is an immediate consequence of the equivalent of (1),
(7) and (8).
\end{proof}

Let $n$ be a non-negative integer. In what follows, we denote by
$\mathcal {G}\textrm{-}\textrm{flat}_{\leqslant n}$ (resp.,
$\mathcal {G}_C\textrm{-}\textrm{flat}_{\leqslant n}$) the class of
modules with finite Gorenstein flat (resp., $\mathcal {G}(\mathcal
{W}_F)$-projective) dimension at most $n$.

\begin{theorem}(Foxby equivalence) Let $\mathcal {F}(R)$ be the
class of flat modules. There are equivalences of categories:
\begin{center}
$\xymatrix@C=100pt@R=20pt{
  \mathcal {F}(R) \ar@{^{(}->}[d] \ar@<0.5ex>[r]^{C\otimes_R-} &\mathcal {F}_C(R) \ar@{^{(}->}[d] \ar@<0.5ex>[l]^{\Hom_R(C,-)}\\
  \mathcal {G}(\mathcal {F})\cap \mathcal {A}_C(R)\ar@{^{(}->}[d] \ar@<0.5ex>[r]^{C\otimes_R-} & \mathcal {G}(\mathcal {W}_F) \ar@{^{(}->}[d]
   \ar@<0.5ex>[l]^{\Hom_R(C,-)} \\
  \mathcal {G}\textrm{-}\textrm{flat}_{\leqslant n}\cap\mathcal
{A}_C(R)\ar@{^{(}->}[d] \ar@<0.5ex>[r]^{C\otimes_R-} & \mathcal
{G}_C\textrm{-}\textrm{flat}_{\leqslant n} \ar@{^{(}->}[d] \ar@<0.5ex>[l]^{\Hom_R(C,-)}\\
  \mathcal
{A}_C(R)\ar@<0.5ex>[r]^{C\otimes_R-} &\mathcal {B}_C(R)
\ar@<0.5ex>[l]^{\Hom_R(C,-)}   }$
\end{center}
\end{theorem}
\begin{proof}
Let $M$ be an $R$-module. It suffices to prove the equivalence of
categories of the third rows in the above diagram.

For the third row, it suffices to prove the case for $n=1$. Assume
that $M\in \mathcal {G}_C\textrm{-}\textrm{flat}_{\leqslant 1}$.
Then there exists a short exact sequence
\begin{center}
$
  0 \longrightarrow G_1 \longrightarrow G_0\longrightarrow M \longrightarrow 0
$
\end{center}
in $R$-Mod with $G_0,G_1 \in \mathcal {G}(\mathcal
{W}_F)$. Since $G_1\in \mathcal {B}_C(R)$ by Theorem 3.4, we have
the following exact sequence in $R$-Mod:
\begin{center}
$
  0\longrightarrow \Hom_R(C,G_1) \longrightarrow \Hom_R(C,G_0)\longrightarrow\textrm{ Hom}_R(C,M)\longrightarrow
  0
$
\end{center}
with
$\Hom_R(C,G_0),\Hom_R(C,G_1)\in\mathcal {G}(\mathcal
{F})\cap\mathcal {A}_C(R)$ by Theorem 3.5. Hence, by Lemma 2.13,
$\textrm{ Hom}_R(C,M)\in \mathcal
{G}\textrm{-}\textrm{flat}_{\leqslant 1}\cap\mathcal {A}_C(R)$.

Conversely, assume that $M \in  \mathcal
{G}\textrm{-}\textrm{flat}_{\leqslant 1}\cap\mathcal {A}_C(R) $.
Then there exists a short exact sequence $0\rightarrow
G_1\rightarrow G_0\rightarrow M\rightarrow0$ in $R$-Mod with
$G_0,G_1\in\mathcal {G}(\mathcal {F})\cap\mathcal {A}_C(R) $. Since
$M\in\mathcal {A}_C(R)$ by Lemma 2.13, $\Tor_R^{\geqslant1}(C,M)=0$.
Thus, there exists a short exact sequence
\begin{center}
$0\longrightarrow C\otimes_RG_1\longrightarrow C\otimes_RG_0\longrightarrow
C\otimes_RM\longrightarrow0 $
\end{center}
in $R$-Mod. By Theorem 3.5,
we know that $C\otimes_RG_0,C\otimes_RG_1\in \mathcal {G}(\mathcal
{W}_F)$. Hence, $C\otimes_RM \in \mathcal
{G}_C\textrm{-}\textrm{flat}_{\leqslant 1}$.
\end{proof}

\section { \bf Stability of Categories}

We start with the following definition.

\begin{definition}
Let $M$ be an $R$-module and $n\geqslant2$ a integer. We say that
$M\in \mathcal {G}^n(\mathcal {W}_F)$ if there exists an exact
sequence
\begin{center}
$\mathbb{G}_\bullet=  \cdots\longrightarrow
G_1\longrightarrow G_0\longrightarrow G^0\longrightarrow
G^1\longrightarrow\cdots$
\end{center}
in $\mathcal{G}^{n-1}(\mathcal {W}_F)$ such that $M \cong$
$\im(G_0\rightarrow G^0) $ and that $\mathbb{G}_\bullet$ is
$\Hom_R(\mathcal {G}^{n-1}(\mathcal {W}_F),-)$ and $\mathcal
{G}^{n-1}(\mathcal {W}_F)^+\otimes_R-$ exact.

Set $\mathcal {G}^0(\mathcal {W}_F)=\mathcal {F}_C(R),~\mathcal
{G}^1(\mathcal {W}_F)= \mathcal {G}(\mathcal {W}_F)$. One can easily
check that there is a contain $\mathcal {G}^n(\mathcal
{W}_F)\subseteq\mathcal {G}^{n+1}(\mathcal {W}_F)$ for all
$n\geqslant0$.

Similarly, we can also define modules which belong to $\mathcal
{G}^n(\mathcal {G}\mathcal {F}_C(R)\cap \mathcal {B}_C(R))$ or
$\mathcal {G}^n(\mathcal {F})$ for $n\geqslant2$.
\end{definition}

\begin{lemma}
$\mathcal {P}_C(R)~\bot~\mathcal {G}^2(\mathcal {W}_F)$ and
$\mathcal {I}_C(R)~\top~\mathcal {G}^2(\mathcal {W}_F)$.
\end{lemma}
\begin{proof}
It follows directly from Lemma 2.4, Proposition 3.2 and the fact
that $\mathcal {P}_C(R)\subseteq\mathcal {G}(\mathcal {W}_F)$ and
$\mathcal {I}_C(R)\subseteq \mathcal {G}(\mathcal {W}_F)^+$.
\end{proof}

\begin{lemma}
Let $R$ be a GF-closed ring. Then $\mathcal {P}_C(R)$ is a
projective generator for $\mathcal {G}(\mathcal {W}_F)$.
\end{lemma}
\begin{proof}
Let $M$ be an $R$-module and $M\in\mathcal {G}(\mathcal {W}_F)$. So
$M\in \mathcal {B}_C(R)$ by Theorem 3.4. Thus, we have a short exact
sequence
\begin{center}
$0\longrightarrow M'\longrightarrow C\otimes_RP\longrightarrow M\longrightarrow0$
\end{center}
in $R$-Mod with $P$ projective. By Corollary 3.6, we know that
$M'\in \mathcal {G}(\mathcal {W}_F)$. On the other hand, it follows
from Proposition 3.2 that $\mathcal {P}_C(R)~\bot~\mathcal
{G}(\mathcal {W}_F)$. Hence, $\mathcal {P}_C(R)$ is a projective
generator for $\mathcal {G}(\mathcal {W}_F)$.
\end{proof}

\begin{lemma}
Let $R$ be a GF-closed ring and let $M$ be an $R$-module which
belongs to $\mathcal {G}^2(\mathcal {W}_F)$. Then $M$ admits a
proper left $\mathcal {P}_C(R)$-resolution.
\end{lemma}
\begin{proof}
It follows directly from the definition of modules which belong to
$\mathcal {G}^2(\mathcal {W}_F)$, Lemma 4.2, Lemma 4.3 and [15,
Lemma 2.2(b)].
\end{proof}

\begin{theorem}
Let $R$ be a GF-closed ring. We have $\mathcal {G}^n(\mathcal
{W}_F)=\mathcal {G}(\mathcal {W}_F)$ for all $n\geqslant1$.
\end{theorem}
\begin{proof}
It suffices to prove the case for $n=2$. Let $M$ be an $R$-module
and $M\in\mathcal {G}^2(\mathcal {W}_F)$. Following from Lemma 4.4,
we have the following exact sequence in $R$-Mod:
\begin{center}
$(\alpha)= \cdots \longrightarrow C\otimes_R P_1 \longrightarrow
C\otimes_R P_0\longrightarrow M \longrightarrow 0 $
\end{center}
with each $P_i$ projective such that $\Hom_R(\mathcal
{P}_C(R),-)$ leaves it exact. By Lemma 2.4, Lemma 2.14 and Lemma
4.2, we get that $\mathcal {I}_C(R)\otimes_R-$ leaves ($\alpha$)
exact as well.

On the other hand, since $M\in\mathcal {G}^2(\mathcal {W}_F)$, there
exists a short exact sequence $
  0 \rightarrow M \rightarrow G \rightarrow M' \rightarrow 0
$ in $R$-Mod with $G\in\mathcal {G}(\mathcal {W}_F)$ and
$M'\in\mathcal {G}^2(\mathcal {W}_F)$. Since $G\in\mathcal
{G}(\mathcal {W}_F)$, there exists a short exact sequence $
  0 \rightarrow G \rightarrow C\otimes_RF^0 \rightarrow G'\rightarrow 0
$ in $R$-Mod with $F^0$ flat and $G'\in\mathcal {G}(\mathcal
{W}_F)$. Then we have the following push-out diagram:
\begin{center}
$\xymatrix{
     & & & 0\ar[d]_{}  & 0 \ar[d]_{} &  \\
     & 0\ar[r]& M\ar@{=}[d]^{} \ar[r]& G\ar[d]\ar[r] &M'\ar[d]\ar[r]&0 \\
& 0\ar[r]& M \ar[r]& C\otimes_RF^0\ar[d]\ar[r] &K\ar[d]\ar[r]&0 \\
&&&G'\ar[d]_{} \ar@{=}[r]^{} & G' \ar[d]_{} \\
    && & 0& 0  &
      }$
\end{center}
Consider the following short exact sequence coming from the middle
row of the above diagram:
\begin{center}
$(\beta)=~~~
  0 \longrightarrow M \longrightarrow C\otimes_RF^0 \longrightarrow K \longrightarrow0
$
\end{center}
From the third column of the above push-out diagram,
we know that $\mathcal {I}_C(R)~\top~K$ by Proposition 3.2 and Lemma
4.2. Thus, $(\beta)$ is $\Hom_R(\mathcal {P}_C(R),-)$ and $\mathcal
{I}_C(R)\otimes_R-$ exact. If we now can construct a short exact
sequence
\begin{center}
$(\eta)=~~~ 0 \longrightarrow K\longrightarrow C\otimes_RF^1
\longrightarrow K' \longrightarrow 0 $
\end{center}
in $R$-Mod with $F^1$ flat and $K'$ a module with the same property
as $K$ (that is, there exists a short exact sequence
$(\mu)=~0\rightarrow M''\rightarrow K'\rightarrow H''\rightarrow0$
in $R$-Mod with $M''\in\mathcal {G}^2(\mathcal {W}_F) $ and
$H''\in\mathcal {G}(\mathcal {W}_F$)), then the following exact
sequence can be constructed recursively:
\begin{center}
$(\gamma)=~~~
\xymatrix@C=20pt@R=5pt{
  0 \ar[r] &  K \ar[r]^{} &~~~C\otimes_RF^1~~~ \ar[rr]^{} \ar@{.>}[dr]_{}
                &  &  ~~~ C\otimes_RF^2 ~~~\ar[r] & \cdots    \\
          &&      & K'  \ar@{.>}[ur]^{} \ar@{.>}[dr]^{}  \\
          && ~~~~~0~~~~~ \ar@{.>}[ur]^{}    && ~~~~~0~~~~~
          }$
\end{center}
From the middle row of the above push-out diagram and $(\mu)$, we
get $\mathcal {P}_C(R)~\bot~K$ and $\mathcal {I}_C(R)~\top~K'$ by
Proposition 3.2, Lemma 2.14 and Lemma 4.2. Then we have that
$(\eta)$ is $\Hom_R(\mathcal {P}_C(R),-)$ and $\mathcal
{I}_C(R)\otimes_R-$ exact. So is $(\gamma)$. Assembling the sequence
$(\alpha),~ (\beta)~ \textrm{and}~ (\gamma)$, we get the following
exact sequence in $R$-Mod:\\\\
$\indent\indent\indent\indent \cdots \longrightarrow C\otimes_RP_1
\longrightarrow C\otimes_RP_0 \longrightarrow C\otimes_RF^0
\longrightarrow C\otimes_RF^1 \longrightarrow\cdots $
\\\\
such that $M\cong
\im(C\otimes_RP_0 \rightarrow C\otimes_RF^0)$ and that
$\Hom_R(\mathcal {P}_C(R),-)$ and $\mathcal {I}_C(R)\otimes_R-$ leave
it exact. It follows that $M \in\mathcal {G}(\mathcal {W}_F)$.

Indeed, since $M'\in\mathcal {G}^2(\mathcal {W}_F)$, there exists a
short exact sequence $
  0\rightarrow M' \rightarrow H \rightarrow M'' \rightarrow 0
$ in $R$-Mod with $H\in\mathcal {G}(\mathcal {W}_F)$ and
$M''\in\mathcal {G}^2(\mathcal {W}_F) $. Now consider the following
push-out diagram:
\begin{center}
$\xymatrix{
     &&  0\ar[d]_{}  & 0 \ar[d]_{} &  \\
      & 0\ar[r]&M'\ar[r]^{} \ar[d]& K\ar[d] \ar[r]& G'\ar@{=}[d]_{}\ar[r]&0\\
      & 0 \ar[r]^{} & H \ar[d]_{} \ar[r]^{} & H'\ar[d]_{} \ar[r]^{} &G' \ar[r]^{} & 0  \\
     &&    M''  \ar[d]\ar@{=}[r]^{} & M''\ar[d]_{}   \\
    &&  0& 0  &
      }$
\end{center}
From the middle row of the above diagram, we know $H'\in\mathcal
{G}(\mathcal {W}_F)$ by Corollary 3.6. Then there exists a short
exact sequence $ 0\rightarrow H' \rightarrow C\otimes_RF \rightarrow
H'' \rightarrow0$ in $R$-Mod with $F$ flat and $H''\in\mathcal
{G}(\mathcal {W}_F)$. Consider another push-out diagram:
\begin{center}
$\xymatrix{
     & & & 0\ar[d]_{}  & 0 \ar[d]_{} &  \\
      & 0\ar[r]& K\ar@{=}[d]^{} \ar[r]& H'\ar[d]\ar[r] &M''\ar[d]\ar[r]&0 \\
& 0\ar[r]& K \ar[r]& C\otimes_RF \ar[d]\ar[r] &K'\ar[d]\ar[r]&0 \\
&&&H''\ar[d]_{} \ar@{=}[r]^{} & H'' \ar[d]_{} \\
    && & 0& 0  &
      }$
\end{center}
It is trivial that the third column in the above diagram is as
desired. This completes our proof.
\end{proof}

The following corollary is an immediate consequence of Theorem 3.4
and Theorem 4.5.

\begin{corollary}
Let $R$ be a GF-closed ring. We have $\mathcal {G}^n(\mathcal
{G}\mathcal {F}_C(R)\cap \mathcal {B}_C(R))=\mathcal {G}\mathcal
{F}_C(R)\cap \mathcal {B}_C(R)$ for all $n\geqslant1$.
\end{corollary}

Note that the next result on the class of Gorenstein flat modules is
of [17, Theorem 4.3] or [3, 1.2] when we set $C=R$.

\begin{corollary}
Let $R$ be a GF-closed ring. We have $\mathcal {G}^n(\mathcal
{F})=\mathcal {G}(\mathcal {F})$ for all $n\geqslant1$.
\end{corollary}

\end{document}